\documentclass[12pt]{amsart}
\usepackage{amsmath,amssymb,amsthm,upref,graphicx,mathrsfs}
\usepackage[top=25mm, bottom=25mm, left=30mm, right=30mm]{geometry}
\usepackage{color,xcolor}
\usepackage[
  colorlinks=true,
  linkcolor=blue,
  citecolor=blue,
  urlcolor=blue]{hyperref}

\numberwithin{equation}{section}
\newtheorem{theorem}{Theorem}[section]

\newtheorem{example}[theorem]{Example}

\numberwithin{equation}{section}

\begin{document}
\title[Some applications of interpolating sequences]{Some applications of interpolating sequences for Banach spaces of analytic functions}
\author{Hamzeh Keshavarzi }

\maketitle
\begin{abstract}
M. J. Beltr\'{a}n-Meneua et al. \cite{beltran1} and E. Jord\'{a} and A. Rodr\'{i}guez-Arenas \cite{jorda} characterized the (uniformly) mean ergodic composition operators on $H^\infty(\mathbb{D})$ and $H^\infty_\nu (\mathbb{D})$, respectively.
In this paper, by using the interpolating sequences, we give other necessary and sufficient conditions for the (uniformly) mean ergodicity of composition operators on these spaces. 
\end{abstract}

\section{\textbf{Introduction}}

Let $H(\mathbb{D})$ be the space of all analytic functions on the open unit disk $\mathbb{D}$.  Let $\varphi$ be an analytic self-map of $\mathbb{D}$. The map $\varphi$ induces a composition operator $C_\varphi$ on $H(\mathbb{D})$ which is defined by $C_\varphi f=f\circ \varphi$. We refer to \cite{cowen1, shapiro} for various aspects of the theory of composition operators on holomorphic function spaces.
 Also, $H^\infty(\mathbb{D})$ is the Banach space of bounded functions in  $H(\mathbb{D})$ with supremum norm.

 Consider the continuous function $\nu:\mathbb{D}\rightarrow (0,\infty)$.
 For $0<p<\infty$, the weighted Bergman space $A^p_\nu(\mathbb{D})$ is the space of analytic functions $f$ on $\mathbb{D}$ where
 \begin{equation*}
\|f\|_{\nu,p}^p= \int_{z\in \mathbb{D}} |f(z)|^p \nu(z)<\infty.
\end{equation*}
 For $p=\infty$, the weighted Bergman space of infinite order is defined as:
\begin{equation*}
A^p_\nu(\mathbb{D})=H^\infty_\nu(\mathbb{D})=\{f\in H(\mathbb{D}): \ \sup_{z\in \mathbb{D}} |f(z)|\nu(z)<\infty\}.
\end{equation*}
Throughout this paper, $\nu$ has the following properties:
\begin{itemize}
\item [(i)] it is radial, that is, $\nu(z)=\nu(|z|)$),
\item [(ii)] it is decreasing,
\item [(iii)] $\lim_{|z|\rightarrow 1} \nu(z)=0$ and
\item [(iv)]  it satisfies the Lusky condition;
$\inf_n \frac{\nu(1-2^{-n-1})}{\nu(1-2^{-n})}>0.$
\end{itemize}

A well-known fact is that $\nu \simeq \tilde{\nu}$, where
\begin{equation*}
\tilde{\nu}(z) =\dfrac{1}{sup\{|f(z)|: \ f\in H^\infty_\nu(\mathbb{D}), \|f\|\leq 1\}}.
\end{equation*}
By \cite[1.2 Properties, Part (iv)]{bier}, for every $z\in \mathbb{D}$ there is a function $f_z$ in the unit ball of $H^\infty_\nu(\mathbb{D})$ such that $f_z(z)=1/\tilde{\nu}(z)$. Bonet et al. \cite{bonet} showed that for such weighted $\nu$, the composition operators generated by disk automorphisms are bounded on $H^\infty_\nu(\mathbb{D})$. Thus, by the Schwarz Lemma, every $C_\varphi$ is bounded on $H^\infty_\nu(\mathbb{D})$. Consider $\alpha>0$;
we use the notation $H^\infty_\alpha(\mathbb{D})$ for the space $H^\infty_\nu(\mathbb{D})$ with the weight $\nu(z)=(1-|z|)^\alpha$.

 Let $\{a_k\}$ be a discrete sequence of points in $\mathbb{D}$. If it is the case that for any bounded sequence of complex number $\{b_k\}$, one can find a bounded analytic function, $f$, in $\mathbb{D}$, such that
$$f(a_k)=b_k, \qquad k=1,2,...,$$
we say that $\{a_k\}$ is an interpolating sequence for $H^\infty(\mathbb{D})$.
Nevannlina \cite{nevanlinna} gave a necessary and sufficient condition for a sequence to be an interpolation sequence in $H^\infty(\mathbb{D})$. Carleson \cite{carleson1} with a simpler condition presented another characterization for these sequences in $H^\infty(\mathbb{D})$.
Using interpolation sequences, Carleson solved the corona conjecture on $H^\infty (\mathbb{D})$ in his celebrated paper \cite{carleson2}.
Berndtsson \cite{bern} gave a sufficient condition for  $H^\infty(\mathbb{B}_n)$-interpolating sequences.

 Consider the sequence
\begin{equation*}
M_n(T)=\dfrac{1}{n}\sum_{j=1}^n T^j,
\end{equation*}
where $T^j$ is the $j-th$ iteration of continuous operator $T$ on the Banach space $X$. We say that $T$ is mean ergodic if $M_n(T)$ converges to a bounded operator acting on $X$, in the strong operator topology. Also, $T$ is called uniformly mean ergodic if $M_n(T)$ converges in the operator norm.
M. J. Beltr\'{a}n-Meneua et al. \cite{beltran1} and E. Jord\'{a} and  A. Rodr\'{i}guez-Arenas \cite{jorda} characterized the (uniformly) mean ergodic composition operators on $H^\infty(\mathbb{D})$ and $H^\infty_\nu (\mathbb{D})$, respectively.
 In this paper, by using the interpolating sequences, we give other necessary and sufficient conditions for the (uniformly) mean ergodicity of composition operators on these spaces.

It is well-known that $H^\infty(\mathbb{D})$ is a Grothendieck Dunford-Pettis (GDP)  space.
Lusky \cite{lusky},  showed that $H^\infty_\nu(\mathbb{D})$ is isomorphic either to a $H^\infty(\mathbb{D})$ or to a $l^\infty$. Hence, $H^\infty_\nu(\mathbb{D})$ also is a GDP space. Lotz \cite{lotz} proved that if $X$ is a GDP space and in addition $\|T^n/n\|\rightarrow 0$, then $T$ is mean ergodic if and only if it is uniformly mean ergodic.

M. J. Beltr\'{a}n-Meneua et al. \cite{beltran1} used the above results to characterize (uniformly) mean ergodic composition operators on $H^\infty(\mathbb{D})$. Indeed, according to the result of Lotz, the mean ergodicity and uniformly mean ergodicity of composition operators on $H^\infty(\mathbb{D})$ are equivalent. Also, they gave another equivalent geometric condition with the mean ergodicity of composition operators. In this paper, we give another equivalent condition: we show that $C_\varphi$ on $H^\infty(\mathbb{D})$ is mean ergodic if and only if $\lim_{n\rightarrow \infty} \| \frac{1}{n} \sum_{j=1}^n C_{\varphi_j}\|_e=0$, where $\|.\|_e$ denotes the essential norm of operators. March T. Boedihardjo and William B. Johnson \cite{boe} investigated the mean ergodicity of operators in Calkin algebra.

E. Jord\'{a} and  A. Rodr\'{i}guez-Arenas \cite{jorda} characterized the (uniformly) mean ergodic composition operators on $H_\nu^\infty(\mathbb{D})$. Also, they showed that if $C_\varphi$ is mean ergodic then $\varphi$ has an interior Denjoy-Wolff point.
In this paper, we give another geometric necessary and sufficient conditions for the (uniformly) mean ergodic composition operators on $H^\infty_\nu(\mathbb{D})$, when $\varphi$ has an interior Denjoy-Wolff point.  Also,  we show that if $\varphi$ is not an elliptic disk automorphism,  then $C_\varphi$ is uniformly mean ergodic on  $H^\infty_\alpha(\mathbb{D})$  if and only if $M_n(C_\varphi)$ converges to $0$ in Calkin algebra of  $H^\infty_\alpha(\mathbb{D})$.

 Throughout the paper, we write $A\lesssim B$ when there is a positive constant $C$ such that $A\leq  CB$, and
$A\simeq B$ when $A\lesssim B$ and $B\lesssim A$.

\section{\textbf{Main Results}}
Let $\varphi$ be an analytic self-map of the unit disk.
The point in the following theorem is called the Denjoy-Wolff point
of $\varphi$.
\begin{theorem}[Denjoy-Wolff Theorem]
If $\varphi$, which is neither the identity nor an elliptic automorphism of $\mathbb{D}$, is an analytic map of the unit disk into itself, then there is a point $w$ in $\overline{\mathbb{D}}$ so that $ \varphi_j\rightarrow w$, uniformly on the compact subsets of $\mathbb{D}$.
\end{theorem}

If $\varphi$ has an interior Denjoy-Wolff point, then we can conjugate it with an analytic self-map of $\mathbb{D}$ that fixes the origin. Hence, if $\varphi$ has an interior Denjoy-Wolff point, we investigate only the case $\varphi(0)=0$.

\begin{theorem} \label{t1}
Consider $\varphi:\mathbb{D}\rightarrow \mathbb{D}$ analytic with $\varphi(0)=0$, which is neither the identity nor an elliptic automorphism of $\mathbb{D}$. Then the following statements are equivalent:
\begin{itemize}
\item[(i)] $C_\varphi$ is mean ergodic on $H^\infty_\nu(\mathbb{D})$.
\item[(ii)] $C_\varphi$ is uniformly mean ergodic on $H^\infty_\nu(\mathbb{D})$.
\item[(iii)] For some (every) sequence $(r_n)\in (0,1)$ where $r_n\uparrow 1$,
\begin{equation} \label{e1}
\lim_{n\rightarrow \infty} \sup_{|z|>r_n} \dfrac{\nu(z)}{\nu(\varphi_n(z))}=0.
\end{equation}
\end{itemize}
Moreover, if $\sup_{z\in \mathbb{D}} \nu(z)\leq 1$, then
\begin{itemize}
\item[(iv)] We have
\begin{equation} \label{e2}
\lim_{n\rightarrow \infty} \sup_{|z|<1}\dfrac{\nu(z)}{\nu(\varphi_n(z))} |\varphi_n(z)|=0.
\end{equation}
\end{itemize}
\end{theorem}

\begin{proof}
(i)$\Leftrightarrow$ (ii):
If $f$ is in the unit ball of  $H^\infty_\nu(\mathbb{D})$, then
\begin{equation*}
|f(z)|\leq \dfrac{1}{\nu(z)}, \qquad \forall z\in  \mathbb{D}.
\end{equation*}
Thus, by using the monotonicity of $\nu$ and the Schwarz lemma
\begin{equation*}
|f(\varphi_n (z))|\nu(z)\leq \dfrac{\nu(z)}{\nu(\varphi_n(z))}\leq 1, \qquad \forall  z\in \mathbb{D}.
\end{equation*}
Hence, $\|C_{\varphi_n}/n\| \rightarrow 0$. Therefore, from \cite{lotz} it follows that (i) and (ii) are equivalent.\\

(ii)$\Rightarrow$(iii). Consider the sequence $(r_n)\in (0,1)$ where $r_n\uparrow 1$ and \ref{e1} does not hold. Since for any $z\in \mathbb{D}$ the sequence $\{|\varphi_i(z)|\}$ is decreasing,
there are an $r>0$ and a sequence $\{a_n\}\subset \mathbb{D}$ where $|a_n|\rightarrow 1$ and
\begin{equation*}
\dfrac{\nu(a_n)}{\nu(\varphi_n(a_n))}\geq r.
\end{equation*}
Since  $|a_n|\rightarrow 1$, and $\lim_{|z|\rightarrow 1} \nu(z)=0$, there are some $s>0$ and $k\in \mathbb{N}$ such that $|\varphi_n(a_n)|\geq s$, for all $n\geq k$. Without loss of generality, we can let $k=1$. Thus, by the Schwarz lemma for $1\leq i \leq n$,
\begin{equation*}
\dfrac{\nu(a_n)}{\nu(\varphi_i(a_n))}\geq\dfrac{\nu(a_n)}{\nu(\varphi_n(a_n))}\geq r, \qquad |\varphi_i(a_n)|\geq |\varphi_n(a_n)|\geq s.
\end{equation*}
Hence, by the proof of \cite[Lemma 13]{cowen2}, for every $n$,  $\{\varphi_i(a_n)\}_{i=1}^n$ is an interpolation. Thus, by \cite[Page 3]{bern} there is some $M>0$ and $\{f_{i,n}\}_{i=1}^n$ in $H^\infty(\mathbb{D})$ such that
\begin{itemize}
\item[(a)] $f_{i,n}(\varphi_i(a_n))=1$ and $f_{i,n}(\varphi_j(a_n))=0$, for $i\neq j$.
\item[(b)] $\sup_{z\in \mathbb{D}} \sum_{i=1}^n |f_{i,n}(z)|\leq M$, for all $n\in \mathbb{N}$.
\end{itemize}
Now we define the functions
\begin{equation*}
f_n(z)=\sum_{i=1}^n z\overline{\varphi_i(a_n)}f_{i,n}(z)f_{\varphi_i(a_n)}(z).
\end{equation*}
 We can easily see that $\sup_{z\in \mathbb{D}} |f_n(z)|(1-|z|)^\alpha \leq M$, $f_n(0)=0$ and
\begin{equation*}
f_n(\varphi_i(a_n))=\dfrac{|\varphi_i(a_n)|^2}{\tilde{\nu}(\varphi_i(a_n))}.
\end{equation*}
Therefore,
\begin{equation*}
\dfrac{1}{n}\sum_{i=1}^n f_n(\varphi_i(a_n))\nu(a_n|) =\dfrac{1}{n}\sum_{i=1}^n \Big( \dfrac{\nu(a_n)}{\tilde{\nu}(\varphi_i(a_n))}\Big) |\varphi_i(a_n)|^2 \gtrsim s^2r.
\end{equation*}
But this is a contradiction with the uniformly mean ergodicity of $C_\varphi$.\\

(iii)$\Rightarrow$(i).
Let $\{r_n\}$ be a sequence in $(0,1)$ such that $r_n\uparrow 1$ and \ref{e1} holds.
Consider $\varepsilon>0$ arbitrary. By using \ref{e1}, there exists some $N$ such that
\begin{equation*}
\sup_{|z|>r_N} \dfrac{\nu(|z|)}{\nu(|\varphi_N(z)|)}<\varepsilon.
\end{equation*}
Now by the Schwarz lemma
\begin{equation*}
\sup_{|z|>r_N} \dfrac{\nu(|z|)}{\nu(|\varphi_n(z)|)}<\varepsilon, \qquad \forall n\geq N.
\end{equation*}
Let $f$ be in the unit ball of $H^\infty_\nu(\mathbb{D})$. Then
\begin{eqnarray*}
\sup_{|z|<1} \dfrac{1}{n} \sum_{i=1}^n |f\circ \varphi_i(z)-f(0)|\nu(|z|) &=& \sup_{|z|\leq r_N} \dfrac{1}{n} \sum_{i=1}^n |f\circ \varphi_i(z)-f(0)|\nu(|z|) \\
&+& \sup_{|z|>r_N} \dfrac{1}{n} \sum_{i=1}^n |f\circ \varphi_i(z)-f(0)|\nu(|z|) .
\end{eqnarray*}
Since $\varphi_i\rightarrow 0$ uniformly on $|z|\leq r_N$, we have
$$\lim_{n\rightarrow\infty} \sup_{|z|\leq r_N} \dfrac{1}{n} \sum_{i=1}^n |f\circ \varphi_i(z)-f(0)|\nu(|z|) =0.$$
Now for $|z|> r_N$: We know that
\begin{equation*}
|f\circ \varphi_i(z)-f(0)|\lesssim \dfrac{1}{\nu(|\varphi_i(z)|)},
\end{equation*}
thus, for $|z|>r_N$ and $n\geq N$,
\begin{eqnarray*}
\dfrac{1}{n} \sum_{i=1}^n |f\circ \varphi_i(z)-f(0)|\nu(|z|) &\lesssim& \dfrac{1}{n} \sum_{i=1}^n \dfrac{\nu(|z|)}{\nu(|\varphi_i(z)|)}\\
&=& \dfrac{1}{n} \Big[ \sum_{i=1}^N \dfrac{\nu(|z|)}{\nu(|\varphi_i(z)|)} + \sum_{i=N+1}^n \dfrac{\nu(|z|)}{\nu(|\varphi_i(z)|)} \Big]\\
&\leq& \dfrac{N}{n}+\dfrac{n-N}{n}\varepsilon.
\end{eqnarray*}
Taking $n$ as sufficiently large, we have $\frac{N}{n}<\varepsilon$. Since $\varepsilon>0$ was arbitrary, (ii) holds.\\

(iii)$\Rightarrow$(iv): Consider $\varepsilon>0$ arbitrary. By using \ref{e1}, there exists some $N$ such that
\begin{equation*}
\sup_{|z|>r_N} \dfrac{\nu(z)}{\nu(\varphi_N(z))}<\varepsilon.
\end{equation*}
Now by the Schwarz lemma
\begin{equation*}
\sup_{|z|>r_N} \dfrac{\nu(z)}{\nu(\varphi_n(z))} |\varphi_n(z)| <\varepsilon, \qquad \forall n\geq N.
\end{equation*}
Also, Since $\varphi_i\rightarrow 0$ uniformly on compact subsets of $\mathbb{D}$,  there exists some $M$ such that
 \begin{equation*}
\sup_{|z|\leq r_N} \dfrac{\nu(z)}{\nu(\varphi_n(z))} |\varphi_n(z)| \leq \sup_{|z|\leq r_N} |\varphi_n(z)| <\varepsilon, \qquad \forall n\geq M.
\end{equation*}
Therefore, for $n\geq \max\{M,N\}$
\begin{equation*}
\sup_{|z|<1}\dfrac{\nu(z)}{\nu(\varphi_n(z))} |\varphi_n(z)| <\varepsilon.
\end{equation*}
This gives the desire result.\\

(iv)$\Rightarrow$(ii): From the proof of  \cite[Proposition 2]{bonet3}, we can conclude that
\begin{equation*}
\|C_{\varphi_i}-K_0\|\simeq \sup_{|z|<1}\max\{ \dfrac{\nu(z)}{\tilde{\nu}(\varphi_i(z))},\nu(z)\} |\varphi_i(z)|.
\end{equation*}
Since $\sup_{z\in \mathbb{D}} \nu(z)\leq 1$, for every $z\in \mathbb{D}$ there is some $g_z$ in the unit ball of $H^\infty_\nu(\mathbb{D})$ where $|g_z(z)|\geq 1$. Thus,
$$\dfrac{1}{\tilde{\nu}(z)} =sup\{|f(z)|: \ f\in H^\infty_\nu(\mathbb{D}), \|f\|\leq 1\}\geq 1,$$
for all $z\in \mathbb{D}$. Therefore,
\begin{eqnarray*}
\|\dfrac{1}{n} \sum_{i=1}^n C_{\varphi_i}-K_0\| &\leq& \dfrac{1}{n} \sum_{i=1}^n \|C_{\varphi_i}-K_0\|\\
&\simeq& \dfrac{1}{n} \sum_{i=1}^n  \sup_{|z|<1}\max\{ \dfrac{\nu(z)}{\tilde{\nu}(\varphi_i(z))},\nu(z)\} |\varphi_i(z)|\\
&\leq& \dfrac{1}{n} \sum_{i=1}^n \sup_{|z|<1}\dfrac{\nu(z)}{\tilde{\nu}(\varphi_i(z))} |\varphi_i(z)|\\
 &\lesssim& \dfrac{1}{n} \sum_{i=1}^n \sup_{|z|<1}\dfrac{\nu(z)}{\nu(\varphi_i(z))} |\varphi_i(z)|\rightarrow 0.
\end{eqnarray*}
as $n\rightarrow \infty$.
\end{proof}

Inequality \ref{e1} implies the following examples.
\begin{example}
Let $\alpha>0$ and $k$ be a positive integer.
If $\varphi(z)=z^k$, then $C_\varphi$ is uniformly mean ergodic on every $H^\infty_\alpha(\mathbb{D})$.
In fact, if $r_n=1-\frac{1}{k^n}$, then \ref{e1} holds:
\begin{eqnarray*}
\lim_{n\rightarrow \infty}  \sup_{|z|>r_n} \dfrac{1-|z|}{1-|\varphi_n(z)|}&=& \lim_{n\rightarrow \infty}  \sup_{|z|>r_n} \dfrac{1-|z|}{1-|z|^{k^n}}= \lim_{n\rightarrow \infty}  \sup_{r>r_n} \dfrac{1-r}{1-r^{k^n}}\\
&=& \lim_{n\rightarrow \infty}  \dfrac{1-(1-\frac{1}{k^n})}{1-(1-\frac{1}{k^n})^{k^n}} \leq \lim_{n\rightarrow \infty}  \dfrac{\frac{1}{k^n}}{1-(1-\frac{1}{k^n})^{k^n}} \\
&=& \dfrac{\lim_{n\rightarrow \infty} \frac{1}{k^n}}{1-e^{-1}}=0.
\end{eqnarray*}
\end{example}

\begin{example}
If $\varphi:\mathbb{D}\rightarrow \mathbb{D}$ does not have a finite angular derivative at any point of the boundary, then $C_\varphi$ is uniformly mean ergodic on $H^\infty_\alpha(\mathbb{D})$. Indeed
\begin{equation*}
\lim_{n\rightarrow \infty}  \sup_{|z|>r_n} \dfrac{1-|z|}{1-|\varphi_n(z)|}\leq \limsup_{|z|\rightarrow 1}  \dfrac{1-|z|}{1-|\varphi(z)|}=0.
\end{equation*}
\end{example}

In Theorems \ref{t3} and \ref{t2} we will use of this fact that: Let $X=H^\infty(\mathbb{D})$ or $H^\infty_\nu(\mathbb{D})$. If $\{f_n\}\subset X$ is a sequence which converges to $0$, uniformly on compact subsets of $\mathbb{D}$, and $Q:X\rightarrow X$ is a compact operator then $\{Q(f_n)\}$  converges to $0$ in the norm of $X$.

\begin{theorem} \label{t3}
Consider $\alpha>0$ and $\varphi:\mathbb{D}\rightarrow \mathbb{D}$ analytic with $\varphi(0)=0$, , which is neither the identity nor an elliptic automorphism of $\mathbb{D}$. Then (in addition to the conditions in Theorem \ref{t1}) another necessary and sufficient condition for the (uniformly) mean ergodicity of $C_\varphi$ on $H^\infty_\alpha(\mathbb{D})$ is that  $\lim_{n\rightarrow \infty} \| \frac{1}{n} \sum_{j=1}^n C_{\varphi_j}\|_e=0$.
\end{theorem}

\begin{proof}
Let $C_\varphi$ be uniformly mean ergodic, so $ \frac{1}{n} \sum_{j=1}^n C_{\varphi_j}\rightarrow K_0$, where $K_0$ is the point evaluation at $0$. We know that $K_0$ is a compact operator, therefore,
\begin{equation*}
\lim_{n\rightarrow \infty} \| \frac{1}{n} \sum_{j=1}^n C_{\varphi_j}\|_e=\lim_{n\rightarrow \infty} \| \frac{1}{n} \sum_{j=1}^n C_{\varphi_j}-K_0\|_e\leq \lim_{n\rightarrow \infty} \| \frac{1}{n} \sum_{j=1}^n C_{\varphi_j}-K_0\|=0.
\end{equation*}

Conversely, we show that Part (iii) of Theorem \ref{t1} holds. Let $T_n=\frac{1}{n} \sum_{j=1}^n C_{\varphi_j}$ and $\{f_m\}$ be a bounded sequence in $H^\infty_\alpha(\mathbb{D})$, where $f_m$ converges to $0$ uniformly on the compact subsets of $\mathbb{D}$. If $Q$ is a compact operator, then
\begin{equation*}
\|T_n+Q\|\gtrsim \lim_{m\rightarrow\infty} \|T_n f_m+Qf_m\|=\lim_{m\rightarrow\infty} \|T_n f_m\|.
\end{equation*}
Thus,
 \begin{equation} \label{e}
  \lim_{n\rightarrow\infty}\lim_{m\rightarrow\infty} \|T_n f_m\|=0.
 \end{equation}
 Let Part (iii) of Theorem \ref{t1} do not hold.
  Consider the sequence $(r_n)\in (0,1)$ where $r_n\uparrow 1$ and $r_n^n\uparrow 1$ (for example, $r_n=\frac{1+1/n}{e^{1/n}}$) but \ref{e1} does not hold. It is suffices to construct a bounded sequence $\{f_m\}$ in $H^\infty_\alpha(\mathbb{D})$ where $f_m$ converges to $0$ uniformly on compact subsets of $\mathbb{D}$, however,  \ref{e} does not hold.
Again there are an $r>0$ and a sequence $\{a_n\}\subset \mathbb{D}$ where $|a_n|\geq r_n$ and
\begin{equation*}
\dfrac{1-|a_n|}{1-|\varphi_n(a_n)|}\geq r.
\end{equation*}
Thus,
\begin{eqnarray*}
r&\leq& \dfrac{1-|a_n|}{1-|\varphi_n(a_n)|}= \dfrac{(1-|a_n|)\sum_{j=0}^{n-1} |\varphi_n(a_n)|^j}{1-|\varphi_n(a_n)|^n}\\
&\leq& \dfrac{(1-|a_n|)\sum_{j=0}^{n-1} |a_n|^j}{1-|\varphi_n(a_n)|^n}=\dfrac{1-|a_n|^n}{1-|\varphi_n(a_n)|^n}
\end{eqnarray*}
Since  $|a_n|^n\rightarrow 1$ there are some $s>0$ and $k\in \mathbb{N}$ such that $|\varphi_n(a_n)|\geq|\varphi_n(a_n)|^n \geq s$, for all $n\geq k$. Without loss of generality, we can let $k=1$. Thus, by the Schwarz lemma  for $1\leq i \leq n$,
\begin{equation*}
\dfrac{1-|a_n|}{1-|\varphi_i(a_n)|}\geq \dfrac{1-|a_n|}{1-|\varphi_n(a_n)|}\geq r, \qquad |\varphi_i(a_n)|^n\geq |\varphi_n(a_n)|^n\geq s
\end{equation*}
Similar to the proof of Theorem \ref{t1}, there are some $M>0$ and $\{f_{i,n}\}_{i=1}^n$ in $H^\infty(\mathbb{D})$ such that
\begin{itemize}
\item[(a)] $f_{i,n}(\varphi_i(a_n))=1$ and $f_{i,n}(\varphi_j(a_n))=0$, for $i\neq j$.
\item[(b)] $\sup_{z\in \mathbb{D}} \sum_{i=1}^n |f_{i,n}(z)|\leq M$.
\end{itemize}
Now we construct the functions $f_m$ as follows:
\begin{equation*}
f_m(z)=\sum_{i=1}^m \dfrac{z^m\overline{\varphi_i(a_m)}^m f_{i,m}(z)}{(1-\overline{\varphi_i(a_m)}z)^\alpha}.
\end{equation*}
It is easily to check that  $\sup_{z\in \mathbb{D}} |f_m(z)|(1-|z|)^\alpha \leq M$, $f_m$ converges to $0$ uniformly on the compact subsets of $\mathbb{D}$, and
\begin{equation*}
f_m(\varphi_i(a_m))=\dfrac{|\varphi_i(a_m)|^{2m}}{(1-|\varphi_i(a_m)|^2)^\alpha}.
\end{equation*}
Therefore, for $m\geq n$
\begin{equation*}
\dfrac{1}{n}\sum_{i=1}^n f_m(\varphi_i(a_m))(1-|a_m|)^\alpha =\dfrac{1}{n}\sum_{i=1}^n \Big( \dfrac{1-|a_m|}{1-|\varphi_i(a_m)|^2}\Big)^\alpha |\varphi_i(a_m)|^{2m} \geq \dfrac{s^2r^\alpha}{2^\alpha}.
\end{equation*}
Thus, Equation \ref{e} does not hold and this is a contradiction.
\end{proof}

The equivalency of Parts (i), (ii), and (iii) of the following theorem has been shown in \cite[Theorem 3.3]{beltran1}. Here we add Part (iv).
\begin{theorem} \label{t2}
Consider $\varphi:\mathbb{D}\rightarrow \mathbb{D}$ analytic with $\varphi(0)=0$, , which is neither the identity nor an elliptic automorphism of $\mathbb{D}$. Then the following statements are equivalent:
\begin{itemize}
\item[(i)] $C_\varphi$ is mean ergodic on $H^\infty (\mathbb{D})$.
\item[(ii)] $C_\varphi$ is uniformly mean ergodic on  $H^\infty (\mathbb{D})$.
\item[(iii)] $\lim_{n\rightarrow \infty} \|\varphi_n\|_{\infty}=0.$
\item[(iv)] $\lim_{n\rightarrow \infty} \| \frac{1}{n} \sum_{j=1}^n C_{\varphi_j}\|_e=0.$
\end{itemize}
\end{theorem}

\begin{proof}
Let $T_n=\frac{1}{n} \sum_{j=1}^n C_{\varphi_j}$ and $\{f_m\}$ be a bounded sequence in $H^\infty(\mathbb{D})$, where $f_m$ converges to $0$ uniformly on the compact subsets of $\mathbb{D}$. Again, we must have
 \begin{equation}\label{e5}
  \lim_{n\rightarrow\infty}\lim_{m\rightarrow\infty} \|T_n f_m\|_\infty=0,
 \end{equation}
Let $\|\varphi_n\|_\infty\nrightarrow 0$, so $\|\varphi_n\|_\infty=1$, for all $n$. We construct a bounded sequence $\{f_m\}$ in $H^\infty(\mathbb{D})$ where $f_m$ converges to $0$ uniformly on the compact subsets of $\mathbb{D}$, but  \ref{e5} does not hold. There is some $r>0$ and $\{a_n\}\subset \mathbb{D}$ such that
$|\varphi_n(a_n)|^{2n}>r$. Hence, by the Schwarz lemma
\begin{equation*}
|\varphi_j(a_n)|^{2n}>r \ \ and  \ \ |\varphi_j(a_n)|>r,   \qquad 1\leq j\leq n.
\end{equation*}
Thus, there is a bounded  sequence $\{g_m\}$ in $H^\infty(\mathbb{D})$ where
\begin{equation*}
g_m(\varphi_i(a_m))=\overline{\varphi_i(a_m)}^m, \qquad 1\leq j\leq m.
\end{equation*}
Consider the functions $f_m(z)=z^mg_m(z)$. Thus, $\{f_m\}$ is a bounded sequence in $H^\infty(\mathbb{D})$ which converges to $0$ uniformly on the compact subsets of $\mathbb{D}$. Also, for $m\geq n$
\begin{equation}
\dfrac{1}{n}\sum_{i=1}^n f_m(\varphi_i(a_m))=\dfrac{1}{n}\sum_{i=1}^n  |\varphi_i(a_m)|^{2m} \geq r.
\end{equation}
This is a contradiction.
\end{proof}

\subsection*{Boundary Denjoy-Wolff point}
In \cite[Theorem 3.6]{beltran1} and \cite[Theorem 3.8]{jorda}, it has shown that if $\varphi$ has boundary Denjoy-Wolff point, then $C_\varphi$ is not mean ergodic on $H^\infty (\mathbb{D})$ and $H^\infty_\nu (\mathbb{D})$, respectively.

\subsection*{Disk automorphisms}
Let $\varphi$ be an automorphism of the disk. If $\varphi$ has a boundary Denjoy-Wolff point then by using the preceding part, $C_\varphi$ failed to be mean ergodic on $H^\infty(\mathbb{D})$ and $H^\infty_\nu(\mathbb{D})$.
If $\varphi$ is an elliptic disk automorphism, then there are some
$\lambda\in
\partial\mathbb{D}$ and some
disk automorphism $\psi$ such that $\psi \circ \varphi\circ \psi^{-1}(z)=\lambda z$.  In this case, Equation \ref{e1} does not hold. But 
\cite[Proposition 18]{wolf1} and \cite[Theorem 3.8]{jorda} imply that:
$C_\varphi$ is (uniformly) mean ergodic on $H^\infty_\nu(\mathbb{D})$ if and only if $\lambda$ is a root of $1$. Also, \cite[Theorem 2.2]{beltran1} gives a similar result on $H^\infty(\mathbb{D})$.

\vspace*{1cm}
 Hamzeh Keshavarzi

E-mail: Hamzehkeshavarzi67@gmail.com

Department of Mathematics, College of Sciences,
Shiraz University, Shiraz, Iran

\end{document}